\documentclass[10pt]{amsart}
\usepackage{amsmath}
\usepackage{amscd}
\usepackage{amssymb}
\usepackage{amsthm}
\RequirePackage{filecontents}
\usepackage{fullpage}
\usepackage[numbers]{natbib}  
\usepackage[draft]{hyperref}

\newenvironment{psmallmatrix}
  {\left(\begin{smallmatrix}}
  {\end{smallmatrix}\right)}

\theoremstyle{plain}
\newtheorem{prop}{Proposition}
\newtheorem{lem}[prop]{Lemma}
\newtheorem{cor}[prop]{Corollary}
\theoremstyle{definition}

\newtheorem{rem}[prop]{Remark}
\newtheorem{ex}[prop]{Example}

\title{Overpartition $M2$-rank differences, class number relations, and vector-valued mock Eisenstein series}
\author{Brandon Williams}

\address{Department of Mathematics \\ University of California \\ Berkeley, CA 94720}

\email{btw@math.berkeley.edu}

\begin{document}

\nocite{*}

\maketitle

\begin{abstract} We prove that the generating function of overpartition $M2$-rank differences is, up to coefficient signs, a component of the vector-valued mock Eisenstein series attached to a certain quadratic form. We use this to compute analogs of the class number relations for $M2$-rank differences. As applications we split the Kronecker-Hurwitz relation into its ``even" and ``odd" parts and calculate sums over Hurwitz class numbers of the form $\sum_{r \in \mathbb{Z}} H(n - 2r^2)$.
\end{abstract}

\section{Introduction and statement of results}

An \textbf{overpartition} of $n \in \mathbb{N}_0$ is a partition in which the first occurance of a number may be overlined (distinguished). Equivalently, an overpartition of $n$ is a decomposition $n = n_1 + n_2$, with $n_1,n_2 \ge 0$, together with a partition of $n_1$ into distinct parts (the overlined numbers) and an arbitrary partition of $n_2$.
For example, the $8$ overpartitions of $n = 3$ are $$3 = \overline{3} = 2 + 1 = \overline{2} + 1 = 2 + \overline{1} = \overline{2} + \overline{1} = 1 + 1 + 1 = \overline{1} + 1 + 1.$$

In \cite{L}, Lovejoy introduced the $M2$-rank statistic $$M2\text{-rank}(\lambda) = \Big\lceil \frac{\ell(\lambda)}{2} \Big\rceil - n(\lambda) + n(\lambda_0) - \chi(\lambda),$$ where $\ell(\lambda)$ denotes the largest part of the overpartition $\lambda$, $n(\lambda)$ is the number of parts, $n(\lambda_o)$ is the number of odd non-overlined parts, and $$\chi(\lambda) = \begin{cases} 1: & \text{the largest part of $\lambda$ is odd and non-overlined}; \\ 0: & \text{otherwise}. \end{cases}$$ Let $M2_e(n)$ and $M2_o(n)$ denote the number of overpartitions of $n$ with even $M2$-rank resp. odd $M2$-rank and define the $M2$-rank difference $$\overline{\alpha}_2(n) = M2_e(n) - M2_o(n).$$ Bringmann and Lovejoy proved \cite{BL} that the generating function $$\overline{f_2}(q) = \sum_{n=0}^{\infty} \overline{\alpha}_2(n) q^n = 1 + 2q + 4q^2 - 2q^4 + 8q^5 + 8q^6 + ...$$ is a mock modular form of weight $3/2$, i.e. the holomorphic part of a harmonic weak Maass form. \\

It was remarked in \cite{W1} that the generating function for the difference counts between overpartitions with even and odd (Dyson's) ranks is, up to the signs of the coefficients, the $\mathfrak{e}_0$-component of a vector valued mock Eisenstein series for the dual Weil representation attached to the quadratic form $Q(x,y,z) = x^2 + y^2 - z^2$ on $\mathbb{Z}^3$ in a sense to be explained in section 2 below. (It is also closely related to the mock Eisenstein series attached to the quadratic form $Q(x) = 2x^2$.) It was observed there that the function $\overline{f_2}$ seemed to be similarly related to the mock Eisenstein series attached to the quadratic form $Q_2(x,y,z) = 2x^2 - y^2 + 2z^2$ (or equivalently $2x^2 + 2y^2 - z^2$). \\

The first purpose of this note is to prove the observation in the previous paragraph.

\begin{prop} The series $$1 - \sum_{n=1}^{\infty} |\overline{\alpha}_2(n)| q^n = 1 - 2q - 4q^2 - 2q^4 - 8q^5 - 8q^6 - ...$$ is the $\mathfrak{e}_0$-component of the mock Eisenstein series for the dual Weil representations attached to the quadratic form $Q_2(x,y,z) = 2x^2 - y^2 + 2z^2$ and to the quadratic form $\tilde Q_2(x,y,z) = 2x^2 - 2y^2 + z^2$.
\end{prop}

As in the examples of \cite{W2}, calculating the other components of the mock Eisenstein series leads to ``class number relations" by comparing coefficients of certain vector-valued modular or quasimodular forms of weight $2$. (This is very similar to the arguments of \cite{BK2} that use mixed mock modular forms to derive relations among class numbers.) For $n \in \mathbb{N}$, let $\sigma_1(n) = \sum_{d | n} d$, $\lambda_1(n) = \frac{1}{2}\sum_{d | n} \min(d,n/d)$ and let $H(n)$ denote the Hurwitz class number. The identities we find in this case are:

\begin{prop} Define $$\alpha(n) = \sum_{r \in \mathbb{Z}} |\overline{\alpha}_2(n-r^2)| + 4 \sum_{r \, \mathrm{odd}} H(4n-r^2) + 4 \lambda_1(n) - \begin{cases} 4: & n = \square; \\ 0: & \mathrm{otherwise}. \end{cases}$$ Then $\alpha(n)$ satisfies the identity $$\alpha(n) = \begin{cases} 8 \sigma_1(n/2) + 16 \sigma_1(n/4): & n \equiv 0 \, \bmod \, 4; \\ 24 \sigma_1(n/2): & n \equiv 2 \, \bmod \, 4; \\ 4 \sigma_1(n): & n \equiv 1 \, \bmod \, 2. \end{cases}$$
\end{prop}

$\overline{\alpha}_2(n)$ can be expressed in terms of Hurwitz class numbers and representation counts by sums of three squares, using Theorem 1.1 of \cite{BL}. We can use this to isolate the term $\sum_{r \, \mathrm{odd}} H(4n-r^2)$ in the identity above. This is a strengthening of Theorem 1 of \cite{six} which considers the case that $n$ is prime.

\begin{cor} For any $n \in \mathbb{N}$, $$\sum_{r \, \mathrm{odd}} H(4n - r^2) = \begin{cases} (2/3) \sigma_1(n): & n \, \mathrm{odd}; \\ 4 \sigma_1(n/2) - 2 \lambda_1(n): & n \equiv 2 \, (\bmod\, 4); \\ (2/3) \sigma_1(n) + 2 \sigma_1(n/2) - (8/3) \sigma_1(n/4) + 4 \lambda_1(n/4) - 2 \lambda_1(n): & n \equiv 0 \, (\bmod \, 4). \end{cases}$$
\end{cor}

By a similar argument to the proof of Proposition 2 we can give another ``class number relation" for $\overline{\alpha}_2$ of a different sort:

\begin{prop} Write $n \in \mathbb{N}$ in the form $n = 2^{\nu} m$ with $m$ odd. Then $$\sum_{r \in \mathbb{Z}} |\overline{\alpha}_2(n - 2r^2)| = 2 \sigma_1(n,\chi) \cdot \Big( 2 + \frac{\chi(m)}{2^{\nu}} \Big) + 4 \sum_{N(\mathfrak{a}) = 2n} \Big( |b| - a \Big) + \begin{cases} 4: & 2n = \square; \\ 0: & \mathrm{otherwise}; \end{cases}$$ where $\chi$ is the Dirichlet character modulo $8$ given by $\chi(1) = \chi(7) = 1$ and $\chi(3) = \chi(5) = -1$; and $\sigma_1(n,\chi)$ is the twisted divisor sum $$\sigma_1(n,\chi) = \sum_{d | n} \chi(n/d) d;$$ and $\mathfrak{a}$ runs through all ideals of $\mathbb{Z}[\sqrt{2}]$ having norm $2n$, and $a + b \sqrt{2} \in \mathfrak{a}$ is a generator with minimal trace $2a > 0$.
\end{prop}

As before, this translates into an identity for Hurwitz class numbers:

\begin{cor} (i) For any $n = 2^{\nu} m$, $m$ odd, $$\sum_{r \in \mathbb{Z}} H(4n - 2r^2) = \frac{2}{3} \sigma_1(n,\chi) \cdot \Big( 2 + \frac{\chi(m)}{2^{\nu + 2}} \Big) + \frac{1}{2} \sum_{N(\mathfrak{a}) = 8n} \Big( |b| - a \Big).$$ (ii) If $n$ is odd, then $$\sum_{r \in \mathbb{Z}} H(2n - 2r^2) = \frac{4 + \chi(n)}{6} \sigma_1(n,\chi) + \frac{1}{2} \sum_{N(\mathfrak{a}) = 4n} \Big( |b| - a \Big).$$ (iii) If $n$ is odd, then $$\sum_{r \in \mathbb{Z}} H(n - 2r^2) = \frac{2 + \chi(n)}{6} \sigma_1(n,\chi) + \frac{1}{2} \sum_{N(\mathfrak{a}) = 2n} \Big( |b| - a\Big).$$
\end{cor}
If there are no ideals of norm $2n$ then the error term $\frac{1}{2} \sum_{N(\mathfrak{a}) = 2n} (|b| - a)$ vanishes. For example, for primes $p$ that remain inert in $\mathbb{Z}[\sqrt{2}]$ (i.e. $p \equiv 3$ or $p \equiv 5$ modulo $8$), we get the formulas 
$$\sum_{r \in \mathbb{Z}} H(4p - 2r^2) = \frac{7}{6} (p-1);\;\; \sum_{r \in \mathbb{Z}} H(2p - 2r^2) = \frac{p-1}{2}; \;\; \sum_{r \in \mathbb{Z}} H(p - 2r^2) = \frac{p-1}{6}.$$


The identities for Hurwitz class number sums of the forms $\sum_{r \, \mathrm{odd}} H(n - r^2)$ and $\sum_{r \in \mathbb{Z}} H(n - 2r^2)$ when $n$ is a multiple of $4$ can be derived from results of Bringmann-Kane \cite{BK2} and Hirzebruch-Zagier \cite{HZ}, respectively, and possibly their techniques could be adapted to general $n$. On the other hand the vector-valued approach in this note seems to make this somewhat easier. (It could also be mentioned that \cite{W3} gives a similar vector-valued interpretation of the two papers mentioned above, although the computation of class number sums there also seems to apply only to $n \equiv 0 \, (4)$.) \\

The rest of this note is organized as follows. We recall some facts about vector-valued modular forms and (mock) Eisenstein series in section 2. In section 3 we compute the mock Eisenstein series attached to a particular quadratic form and show that Bringmann and Lovejoy's form $\overline{f_2}$ occurs up to coefficient signs as a component of that series. In sections 4 and 5 we transform that mock Eisenstein series into a vector-valued mixed mock modular form essentially by multiplying by a theta function and prove Propositions 2 and 4 by holomorphic projection. This is similar to the proof technique of \cite{ORR} which finds relations among the coefficients of some of Ramanujan's mock theta functions. The procedure is more or less the same for all mock Eisenstein series attached to quadratic forms and we can use the formula of section 7 of \cite{W2} to make the computation rather short.

\section{Notation and Background}

The purpose of this section is to explain what is meant by a mock Eisenstein series attached to a quadratic form and to fix some notation for the rest of this note. For convenience we will use the notation $\mathbf{e}(x) =e^{2\pi i x}$. \\

Suppose that $Q$ is an integral quadratic form in $e$ variables with Gram matrix $\mathbf{S}$; that is, $\mathbf{S}$ is symmetric and integral with even diagonal and $Q(x) = \frac{1}{2}x^T \mathbf{S}x$. Let $A = A(Q) = \mathbf{S}^{-1} \mathbb{Z}^e / \mathbb{Z}^e$ denote the discriminant group of $\mathbf{S}$, and let $(b^+,b^-)$ denote the numbers of positive and negative eigenvalues of $\mathbf{S}$ (so the signature of $Q$ is $\mathrm{sig}(Q) = b^+ - b^-$). There is a representation $\rho^*_Q$ (denoted $\rho^*$ if $Q$ is clear from the context) of the metaplectic group $\tilde \Gamma = Mp_2(\mathbb{Z})$ on the group ring $\mathbb{C}[A]$ given by $$\rho^*(S) \mathfrak{e}_{\gamma} = \frac{\mathbf{e}(\mathrm{sig}(Q)/8)}{\sqrt{|A|}} \sum_{\beta \in A} \mathbf{e}(\gamma^T \mathbf{S} \beta) \mathfrak{e}_{\beta}, \; \; \rho^*(T) \mathfrak{e}_{\gamma} = \mathbf{e}(-Q(\gamma)) \mathfrak{e}_{\gamma}, \; \; \gamma \in A,$$ where $\mathfrak{e}_{\gamma}, \; \gamma \in A$ form the natural basis of $\mathbb{C}[A]$ and where $S = (\begin{psmallmatrix} 0 & -1 \\ 1 & 0 \end{psmallmatrix}, \sqrt{\tau})$, $T= (\begin{psmallmatrix} 1 & 1 \\ 0 & 1 \end{psmallmatrix}, 1)$ are the standard generators of $\tilde \Gamma$. This representation arises naturally through the action of $\tilde \Gamma$ on the (complex conjugate of the) theta function of $Q$ in the following sense. For any $b^+$-dimensional subspace $z \subseteq \mathbb{R}^e$ on which $Q$ is positive-definite, write $\lambda_z$ and $\lambda_{z^{\perp}}$ for the projections of $\lambda \in \mathbb{R}^e$ onto $z$ and its orthogonal complement (with respect to $Q$) $z^{\perp}$. Then the theta function $$\Theta_z(\tau) = \sum_{\lambda \in \mathbf{S}^{-1} \mathbb{Z}^e} \mathbf{e}\Big(Q(\lambda_z) \tau + Q(\lambda_{z^{\perp}}) \overline{\tau}\Big) \mathfrak{e}_{\lambda}$$ satisfies $$\overline{\Theta_z(M \cdot \tau)} = (c \tau + d)^{b^-/2} (c \overline{\tau} + d)^{b^+/2} \rho^*(M) \overline{\Theta_z(\tau)}$$ for all $M = (\begin{psmallmatrix} a & b \\ c & d \end{psmallmatrix}, \sqrt{\tau}) \in \tilde \Gamma$. See \cite{B}, chapter 4. \\

For any weight $k \in 2 \mathbb{Z} - \frac{\mathrm{sig}(Q)}{2}$ one can define the Eisenstein series of weight $k$ attached to $Q$: $$E_k^*(\tau,s;Q) = \sum_{M \in \tilde \Gamma_{\infty} \backslash \tilde \Gamma} (y^s \mathfrak{e}_0) \Big|_{k,\rho^*} M, \; \; \mathrm{Re}[s] \; \text{large enough},$$ where $|_{k,\rho^*}$ is a Petersson slash operator twisted by the representation $\rho^*$, i.e. $$f|_{k,\rho^*} M = (c \tau + d)^{-k} \rho^*(M)^{-1} f(M \cdot \tau), \; \; M \in \tilde \Gamma,$$ and $\tilde \Gamma_{\infty}$ is the subgroup of those $M \in \tilde \Gamma$ for which $(y^s \mathfrak{e}_0) |_{k,\rho^*} M = y^s \mathfrak{e}_0$ (i.e. the subgroup generated by $T$ and by the center of $\tilde \Gamma$.) This series can be meromorphically extended to all $s \in \mathbb{C}$, and the extension is entire unless the weight is $k \in \{0,\pm 1/2, \pm 5/2, \pm 9/2,...\}$ in which case $E_k^*(\tau,s;Q)$ may have at worst a simple pole at the point $s = 1-k$ and nowhere else. In particular the zero-value $E_k^*(\tau;Q) = E_k^*(\tau,0,Q)$ is always well-defined. \\

Since $E_k^*(\tau,s;Q)$ satisfies $$\Delta_k E_k^*(\tau,s;Q) = s (s - k - 1) E_k^*(\tau,s;Q)$$ under the weight $k$ Laplacian $$\Delta_k = 4y^2 \frac{\partial^2}{\partial \tau \partial \overline{\tau}} - 2iky \frac{\partial}{\partial \overline{\tau}},$$ it follows that $E_k^*(\tau;Q)$ is annihilated by $\Delta_k$, i.e. it is a harmonic Maass form. Expanding $E_k^*(\tau;Q)$ as a Fourier series and using the fact that each term in the series is annihilated by $\Delta_k$ we find an expression of the form $$E_k^*(\tau;Q) = \sum_{n \ge 0} a_n q^n + b_0 \phi(y) - \sum_{n > 0}^{\infty} b_n (4\pi n)^{1-k} \Gamma(1-k,4\pi n y) q^{-n},$$ where as always $q = e^{2\pi i \tau}$; and $\Gamma(a,s)$ is the incomplete Gamma function, $\phi(y)$ is an antiderivative of $y^{-k}$, and $a_n,b_n \in \mathbb{C}[A]$ are vectors (Fourier coefficients). One can show that the \textbf{shadow} $$2iy^k \overline{\frac{\partial}{\partial \overline{\tau}} E_k^*(\tau;Q)} = \sum_{n=0}^{\infty} b_n q^n$$ is a modular form of weight $2-k$ for the representation $\rho_{-Q}^*$ attached to $-Q$; when the shadow is nonzero the ``holomorphic part" $$E_k(\tau;Q) = \sum_{n=0}^{\infty} a_n q^n$$ is called the \textbf{mock Eisenstein series} of weight $k$ associated to $Q$. See sections 4,5 of \cite{O} for a detailed overview (extending the proofs there to vector-valued mock modular forms is straightforward). \\

We will need to consider weights $k = 3/2$ and $k=2$. In weight two the shadow must be constant and therefore $$E_2^*(\tau;Q) = \sum_{n \ge 0} a_n q^n + \frac{v}{y}$$ for some vector $v \in \mathbb{C}[A]$ which is invariant under $\rho$. In particular the holomorphic part $$E_2(\tau;Q) = \sum_{n \ge 0} a_n q^n$$ is a \textbf{quasimodular form} as defined in \cite{KZ}, i.e. there exists $d \in \mathbb{N}_0$ such that $E_2(\tau;Q)$ is the holomorphic part of a real-analytic modular form whose Fourier coefficients are polynomials in $y^{-1}$ of degree $\le d$.  (More familiar examples of vector-valued quasimodular forms are the theta functions $\sum_{x \in \mathbf{S}^{-1} \mathbb{Z}^e} p(x) q^{Q(x)} \mathfrak{e}_x$ attached to positive-definite quadratic forms $Q$ and homogeneous polynomials $p(x)$, which are true modular forms only when $p$ is harmonic with respect to $Q$.) The shadow $v \in \mathbb{C}[A]$ of $E_2$ can only be nonzero if $\#A$ is a perfect square (for example by section 6 of \cite{W1}). \\

In weight $3/2$ the mock Eisenstein series for the quadratic form $Q(x) = x^2$ is the generating function of Hurwitz class numbers, $$E_{3/2}(\tau) = 1 - 12 \sum_{n=1}^{\infty} H(n) q^{n/4} \mathfrak{e}_{n/2},$$ which has shadow $-24 \vartheta(\tau)$ where $\vartheta(\tau) = \sum_{n \in \mathbb{Z}} q^{n^2 / 4} \mathfrak{e}_{n/2}$ is the theta function of $Q$. The shadow of any weight 3/2 mock Eisenstein series will be a sum of theta functions by the Serre-Stark basis theorem; this essentially means that all weight $3/2$ mock Eisenstein series differ from true modular forms by generating functions of some linear combinations of Hurwitz class numbers. \\

It is often more convenient to split up the Fourier expansions of vector-valued modular forms into components; so we will usually write out a modular form $f$ as $$f(\tau) = \sum_{\gamma \in A} \sum_{n \in \mathbb{Z} - Q(\gamma)} c(n,\gamma) q^n \mathfrak{e}_{\gamma}$$ and label the individual components by $f_{\gamma}(\tau) = \sum_{n \in \mathbb{Z} - Q(\gamma)} c(n,\gamma) q^n.$ The same rule applies to harmonic Maass forms and mock modular forms.

\section{Proof of proposition 1}

We first observe that $\overline{\alpha}_2(n)$ is positive when $n \equiv 1,2 \, (\bmod \, 4)$ and $\overline{\alpha}_2(n)$ is negative or $0$ when $n \equiv 0,3 \, (\bmod \, 4)$ (with $n=0$ as an exception). This is a consequence of equation (1.5) of \cite{BL} which gives an exact expression for $\overline{\alpha}_2(n)$ in terms of Hurwitz class numbers and representation counts by sums of three squares. It follows that the series in Proposition 1 is \begin{align*} f(\tau) &= 1 - \sum_{n=1}^{\infty} |\overline{\alpha}_2(n)| q^n \\ &= 1 - 2q - 4q^2 - 2q^4 - 8q^5 - 8q^6 - ... \\  &= \frac{1+i}{2} \overline{f_2}(\tau + 1/4) + \frac{1-i}{2} \overline{f_2}(\tau - 1/4),\end{align*} where $\overline{f_2}(\tau) = \sum_{n=0}^{\infty} \overline{\alpha}_2(n) q^n$ and where $q = e^{2\pi i \tau}$. \\

Theorem 2.1 of \cite{BL} shows that the real-analytic correction of $\overline{f_2}(\tau)$ is the harmonic weak Maass form \begin{equation}\begin{split} \overline{f_2}(\tau) - \mathcal{N}_2(\tau) &:= \overline{f_2}(\tau) + \frac{i}{\pi \sqrt{2}} \int_{-\overline{\tau}}^{i \infty} \frac{\Theta(t)}{(-i (\tau + t))^{3/2}} \, \mathrm{d}t \\ &= \overline{f_2}(\tau) - \frac{\sqrt{i}}{\pi \sqrt{2}} \int_{-\overline{\tau}}^{\infty} \frac{\Theta(t)}{(\tau + t)^{3/2}} \, \mathrm{d}t \end{split}\end{equation} of level 16, where $\Theta$ is the classical theta series $\Theta(\tau) = 1 + 2q + 2q^4 + 2q^9 + ...$ and $\mathcal{N}_2$ is a sort of period integral of $\Theta$. It follows that the real-analytic correction of $f(\tau)$ is also \begin{align*} &\quad f(\tau) - \frac{1+i}{2} \mathcal{N}_2(\tau + 1/4) - \frac{1-i}{2} \mathcal{N}_2(\tau - 1/4) \\ &= f(\tau) - \frac{\sqrt{i}}{\pi \sqrt{2}} \int_{-\overline{\tau}}^{i \infty} \frac{1}{( \tau + t)^{3/2}} \Big[ \frac{1+i}{2} \Theta(t - 1/4) + \frac{1-i}{2} \Theta(t+1/4) \Big] \, \mathrm{d}t \\ &= f(\tau) - \frac{\sqrt{i}}{\pi \sqrt{2}} \int_{-\overline{\tau}}^{i \infty} \frac{\Theta(t)}{( \tau + t)^{3/2}} \, \mathrm{d}t, \end{align*} using the fact that $\frac{1+i}{2} \Theta(t-1/4) + \frac{1-i}{2} \Theta(t+1/4) =\Theta(t)$ since the only nonzero Fourier coefficients of $\Theta$ occur with exponents $n \equiv 0,1 \, (\bmod \, 4).$

\begin{rem} This implies that $$\overline{f_2}(\tau) - f(\tau) = 2\sum_{n \equiv 1,2 \, (4)} \overline{\alpha}_2(n) q^n = 4q + 8q^2 + 16q^5 + 16q^6 + ...$$ is a weight-$3/2$ modular form of level $16$. This is true; for example, we can write it as a difference of theta series, $$4q + 8q^2 + 16q^5 + 16q^6 + ... = \vartheta_{B_1}(\tau) - \vartheta_{B_2}(\tau)$$ for the matrices $B_1 = \mathrm{diag}(2,2,2)$ and $B_2 = \mathrm{diag}(4,4,2),$ which can be verified by computing finitely many coefficients. Here the theta functions are $$\vartheta_B(\tau) = \sum_{v \in \mathbb{Z}^3} q^{v^T B v / 2}, \; \; B \in \{B_1,B_2\}.$$
\end{rem}

We need to compare the weak Maass form above with the real-analytic correction of the mock Eisenstein series $E_{3/2}(\tau;Q_2)$ for $Q_2(x,y,z) = 2x^2 - y^2 + 2z^2$. Following remark 16 of \cite{W1}, this is $$E_{3/2}^*(\tau,0;Q_2) = E_{3/2}(\tau;Q_2) + \frac{\sqrt{2i}}{16\pi} \int_{-\overline{\tau}}^{i \infty} \frac{\vartheta(t)}{(t + \tau)^{3/2}} \, \mathrm{d}t,$$ where $\vartheta$ is the shadow $$\vartheta(\tau) = \sum_{\gamma \in A(Q_2)} \sum_{\substack{n \in \mathbb{Z} - Q_2(\gamma) \\ n \le 0}} a(n,\gamma) q^{-n} \mathfrak{e}_{\gamma}.$$ The coefficients $a(n,\gamma)$ are given by $$a(n,\gamma) = -\frac{48 \sqrt{2}}{\sqrt{|A(Q_2)|}} \prod_{\mathrm{bad}\, p} \lim_{s \rightarrow 0} \frac{1 - p^{-2s}}{1 + p^{-1}} L_p(n,\gamma,2 + 2s) \times \begin{cases} 1: & n < 0; \\ 1/2: & n = 0; \end{cases}$$ where the ``bad primes" are $p = 2$ and all primes dividing the discriminant $|A(Q_2)|$ or at which $n$ has nonzero valuation, and $L_p$ is the series $$L_p(n,\gamma,s) = \sum_{\nu=0}^{\infty} p^{-\nu s} \# \Big\{ v \in (\mathbb{Z} / p^{\nu} \mathbb{Z})^3 : \; Q_2(v - \gamma) + n \equiv 0 \; (\bmod \, p^{\nu}) \Big\}.$$

We only need to compute a small number of these coefficients to identify $\vartheta$ among modular forms of weight $1/2$ for the Weil representation (not its dual) attached to the quadratic form $\tilde Q_2$. A general result of Skoruppa \cite{Sk} implies that the space of such forms is spanned by the theta series \begin{align*} \vartheta_1(\tau) &= \Big( 1 + 2q + 2q^4 + 2q^9 + ... \Big) (\mathfrak{e}_{(0,0,0)} + \mathfrak{e}_{(1/4,1/2,1/4)} + \mathfrak{e}_{(1/2,0,1/2)} + \mathfrak{e}_{(3/4,1/2,3/4)}) \\ & \quad + \Big( 2q^{1/4} + 2q^{9/4} + 2q^{25/4} + ... \Big) (\mathfrak{e}_{(1/4,0,3/4)} + \mathfrak{e}_{(1/2,1/2,0)} + \mathfrak{e}_{(3/4,0,1/4)} + \mathfrak{e}_{(0,1/2,1/2)}) \end{align*} and \begin{align*} \vartheta_2(\tau) &= \Big( 1 + 2q + 2q^4 + 2q^9 + ... \Big) (\mathfrak{e}_{(0,0,0)} + \mathfrak{e}_{(1/4,1/2,3/4)} + \mathfrak{e}_{(1/2,0,1/2)} + \mathfrak{e}_{(3/4,1/2,1/4)}) \\ &\quad + \Big( 2q^{1/4} + 2q^{9/4} + 2q^{25/4} + ... \Big) (\mathfrak{e}_{(1/4,0,1/4)} + \mathfrak{e}_{(1/2,1/2,0)} + \mathfrak{e}_{(3/4,0,3/4)} + \mathfrak{e}_{(0,1/2,1/2)}). \end{align*}

Therefore, it is enough to compute the constant term of the shadow $\vartheta(\tau)$. The only bad prime at $n=0$ is $p = 2$, and the corresponding local $L$-functions are \begin{equation} L_2(0,\gamma,s) = \begin{cases} (2^{4s} - 2^{2s+3} + 2^7)(1 - 2^{2-s})^{-1}(2^{2s} - 8)^{-1}: & \gamma = (0,0,0); \\ (1 + 2^{4 - 2s})(1 - 2^{2-s})^{-1}: & \gamma = (1/2,1/2,0); \\ (1 - 2^{2-s})^{-1}: & \text{otherwise}; \end{cases} \end{equation} such that $$\lim_{s \rightarrow 0} (1 - 2^{-2s}) L_2(0,\gamma,2+2s) = \begin{cases} 2: & \gamma = (0,0,0) \; \mathrm{or} \; \gamma = (1/2,1/2,0); \\ 1: & \text{otherwise}. \end{cases}$$ This implies that the shadow is \begin{align*} \vartheta(\tau) &= -8 (\mathfrak{e}_{(0,0,0)} +\mathfrak{e}_{(1/2,1/2,0)}) - 4 (\mathfrak{e}_{(1/4,1/4,1/2)} + \mathfrak{e}_{(3/4,3/4,1/2)} + \mathfrak{e}_{(1/4,3/4,1/2)} + \mathfrak{e}_{(3/4,1/4,1/2)}) + O(q^{1/4}) \\ &= -4 \vartheta_1(\tau) - 4 \vartheta_2(\tau) \end{align*} with $\mathfrak{e}_0$-component $-8\Theta(\tau)$, so the $\mathfrak{e}_0$-component of $E_{3/2}^*(\tau,0;Q_2)$ is $$E_{3/2}^*(\tau,0;Q_2)_0 = E_{3/2}(\tau;Q_2)_0 - \frac{\sqrt{i}}{\pi \sqrt{2}} \int_{-\overline{\tau}}^{i \infty} \frac{\Theta(t)}{(t + \tau)^{3/2}} \, \mathrm{d}t,$$ with the same real-analytic part as $f(\tau).$ It follows that $E_{3/2}^*(\tau,0;Q_2)_0 - f(\tau)$ is a modular form of level $16$. We can verify that it is identically $0$ by comparing Fourier coefficients up to the Sturm bound \cite{St} (i.e. if the $q$-expansion of a modular form $h$ of weight $k$ for a subgroup $\Gamma$ vanishes to order $\ell > \frac{k}{12}[PSL_2(\mathbb{Z}) : \Gamma]$ and $\eta$ is Dedekind's eta function, then $\eta^{-24\ell} \prod_{M \in PSL_2(\mathbb{Z})/\Gamma} h(M \cdot \tau)$ is a holomorphic modular form of negative weight and therefore identically zero.) \\

The case of the quadratic form $\tilde Q_2(x,y,z) = 2x^2 - 2y^2 + z^2$ is similar; here, the space of weight $1/2$ modular forms is spanned by the theta functions \begin{align*} \tilde \vartheta_1(\tau) &= \Big( 1 + 2q + 2q^4 + 2q^9 + ... \Big) (\mathfrak{e}_{(0,0,0)} + \mathfrak{e}_{(1/4,1/4,0)} + \mathfrak{e}_{(1/2,1/2,0)} + \mathfrak{e}_{(3/4,3/4,0)}) \\ &\quad + \Big( 2q^{1/4} + 2q^{9/4} + 2q^{25/4} + ... \Big) (\mathfrak{e}_{(0,0,1/2)} + \mathfrak{e}_{(1/4,1/4,1/2)} + \mathfrak{e}_{(1/2,1/2,1/2)} + \mathfrak{e}_{(3/4,3/4,1/2)}) \end{align*} and \begin{align*} \tilde \vartheta_2(\tau) &= \Big( 1 + 2q + 2q^4 + 2q^9 + ... \Big) (\mathfrak{e}_{(0,0,0)} + \mathfrak{e}_{(1/4,3/4,0)} + \mathfrak{e}_{(1/2,1/2,0)} + \mathfrak{e}_{(3/4,1/4,0)}) \\ &\quad + \Big( 2q^{1/4} + q^{9/4} + 2q^{25/4} + ... \Big) (\mathfrak{e}_{(0,0,1/2)} + \mathfrak{e}_{(1/4,3/4,1/2)} + \mathfrak{e}_{(1/2,1/2,1/2)} + \mathfrak{e}_{(3/4,1/4,1/2)}), \end{align*} again by \cite{Sk}. The local $L$-functions for $n = 0$ are exactly the same as those of equation (2) (although ``otherwise'' now stands for a different set of elements $\gamma$). It follows that $\mathfrak{e}_0$ component of $E_{3/2}^*(\tau,0;\tilde Q_2)$ is the same as the $\mathfrak{e}_0$-component of $E_{3/2}^*(\tau,0;Q_2)$ and therefore also equals $f(\tau).$ \\

\section{Proof of proposition 2}

The real-analytic Eisenstein series $E_{3/2}^*(\tau,0;\tilde Q_2)$ (where $\tilde Q_2(x,y,z) = 2x^2 - 2y^2 + z^2$) corresponds via the theta decomposition to a real-analytic Jacobi form $E_{2,1,0}^*(\tau,z,0;Q)$ for the dual Weil representation $\rho^*$ attached to the quadratic form $Q(x,y) = 2x^2 - 2y^2$, and we obtain a quasimodular form $Q_{2,1,0}$ (the ``Poincar\'e square series of index $1$'') by projecting $E_2^*(\tau,0,0;Q)$ orthogonally to the space of cusp forms $S_2(\rho^*)$ for $\rho^*$ and then adding to it the quasimodular Eisenstein series $E_2(\tau;Q)$. Since $S_2(\rho^*) = \{0\}$, it follows that $Q_{2,1,0}(\tau) = E_2(\tau;Q).$ \\

The $\mathfrak{e}_0$-component of the series $E_2(\tau;Q)$ is $$E_2(\tau;Q)_0 = 1 -4q - 24q^2 - 16q^3 - 40q^4 - 24q^5 - ...$$ Using the structure of quasimodular forms of higher level (proposition 1(b) of \cite{KZ}) we identify it as $$1 - 4q - 24q^2 - 16q^3 - 40q^4 - ... = \frac{2}{3} E_2(2\tau) + \frac{2}{3} E_2(4\tau) - \frac{1}{3} E_2(2\tau + 1/2) - 4 \sum_{n=0}^{\infty} \sigma(2n+1)q^{2n+1},$$ in which the coefficient of $q^n$ is $$\alpha(n) = \begin{cases} -4 \sigma_1(n) : & n \equiv 1\, (2); \\ -24 \sigma_1(n/2): & n \equiv 2\, (4); \\ -8\sigma_1(n/2) - 16 \sigma_1(n/4): & n \equiv 0 \, (4). \end{cases}$$

In general, if $Q$ is an $e$-dimensional quadratic form of signature $\mathrm{sig}(Q) \equiv 0 \, (4)$, let $\tilde Q(x,y) = Q(x) + y^2$. Suppose the weight-$3/2$ mock Eisenstein series for $\tilde Q$ has coefficients denoted $$E_{3/2}(\tau;\tilde Q) = \sum_{\tilde \gamma \in A(\tilde Q)} \sum_{n \in \mathbb{Z} - \tilde Q(\tilde \gamma)} c(n,\tilde \gamma) q^n \mathfrak{e}_{\tilde \gamma}$$ with shadow $$\vartheta(\tau) = \sum_{\gamma \in A(\tilde Q)} \sum_{\substack{n \in \mathbb{Z} - \tilde Q(\tilde \gamma) \\ n \le 0}} a(n,\tilde \gamma) q^{-n} \mathfrak{e}_{\tilde \gamma}.$$ The results of \cite{W2} imply that the coefficient formula for $$Q_{2,1,0}(\tau) = \sum_{\gamma \in A(Q)} \sum_{n \in \mathbb{Z} - Q(\gamma)} b(n,\gamma) q^n \mathfrak{e}_{\gamma}$$ is $$b(n,\gamma) = \sum_{r \in \mathbb{Z}} c(n - r^2/4, (\gamma, r/2)) + \frac{1}{8} \sum_{r \in \mathbb{Z}} a(n - r^2 / 4,(\gamma,r/2)) \Big( |r| - \sqrt{r^2 - 4n} \Big).$$

In particular, the formula for $b(n,0)$ involves knowing both the components of $\mathfrak{e}_{(0,0,0)}$ and $\mathfrak{e}_{(0,0,1/2)}$ in the mock Eisenstein series $E_{3/2}(\tau;\tilde Q_2).$ Computing the first few coefficients gives $$E_{3/2}(\tau;\tilde Q_2)_{(0,0,1/2)} = -4q^{7/4} - 8q^{15/4} - 12q^{23/4} - 12q^{31/4} - ...$$

\begin{lem} Let $H(n)$ denote the Hurwitz class number of $n \in \mathbb{N}_0$; then $$E_{3/2}(\tau;\tilde Q_2)_{(0,0,1/2)} = -4 \sum_{n \equiv 7\, (8)} H(n) q^{n/4}.$$
\end{lem}
\begin{proof} In general, the formula of \cite{BK},\cite{W1} for the weight $3/2$ mock Eisenstein series for a quadratic form $Q$ of signature $(b^+,b^-)$ and dimension $e = b^+ - b^-$ (under the assumption $3 + b^+ - b^- \equiv 0 \, (4)$) implies that the coefficient $c(n,\gamma)$ of $E_{3/2}(\tau;Q)$ is \begin{equation} \begin{split} &\frac{24 (-1)^{(3 + b^+ - b^-)/4} L(1,\chi_{\mathcal{D}})}{\sqrt{\frac{1}{2n}|A(Q)| }\pi} \cdot \Big[ \sum_{d | f} \mu(d) \chi_{\mathcal{D}}(d) d^{-1} \sigma_{-1}(f/d) \Big] \times \\ &\quad\quad\quad\quad \quad \quad \times \prod_{p | ( 2|A(Q)|)} \Big[ \lim_{s \rightarrow 0} \frac{1 - p^{(e-3)/2 - 2s}}{1 - p^{-2}} L_p(n,\gamma,1/2 + e/2 + 2s) \Big], \end{split} \end{equation} where $\mathcal{D}$ is a discriminant defined in Theorem 4.5 of \cite{BK}, and $\chi_{\mathcal{D}}(d) = \left( \frac{\mathcal{D}}{d} \right)$ is the Kronecker symbol, and $\sigma_{-1}(n) = \sum_{d | n} d^{-1}$, and $f^2$ is the largest square dividing $n$ but coprime to $2 \cdot |A(Q)|.$ Finally, recall that $L_p(n,\gamma,s)$ denotes the series $$L_p(n,\gamma,s) = \sum_{\nu=0}^{\infty} p^{-\nu s} \# \Big\{ v \in (\mathbb{Z}/p^{\nu} \mathbb{Z})^e: \; Q(v - \gamma) + n \equiv 0 \, (\bmod \, p^{\nu}) \Big\}.$$ For a fixed lattice, the series $L_p(n,\gamma,s)$ can always be evaluated in closed form in both $n$ and $s$ (for example, using the $p$-adic generating functions of \cite{CKW}) although the result tends to be messy (particularly for $p = 2$). \\

In our case, most of these terms are the same between the quadratic forms $Q(x) = x^2$ (in which case, $E_{3/2}$ is essentially the Zagier Eisenstein series and $c(n,\gamma)$ is $-12H(4n)$) and $\tilde Q_2(x,y,z) = 2x^2 - 2y^2 + z^2$. The only differences are the order of the discriminant group and the local $L$-functions above the prime $p=2$. \\

For $Q(x) = x^2$, the result is that $$\lim_{s \rightarrow 0} (1 - 2^{-2s}) L_2\Big( \frac{8n+7}{4}, \frac{1}{2}, 1 + 2s \Big) = \lim_{s \rightarrow 0} \frac{(1 - 2^{-2s})(2^{1+2s} + 1)}{2^{1+2s} - 1} = \frac{3}{2},$$ while for $\tilde Q_2(x,y,z) = 2x^2 - 2y^2 + z^2$ the result is that $$\lim_{s \rightarrow 0} (1 - 2^{-2s}) L_2\Big( \frac{8n+7}{4}, \frac{1}{2}, 2+2s \Big) = \lim_{s \rightarrow 0} \frac{(1 - 2^{-2s}) (2^{2+2s} + 4)}{2^{2+2s} - 4} = 2$$ and $$\lim_{s \rightarrow 0} (1 - 2^{-2s}) \underbrace{ L_2\Big( \frac{8n+3}{4}, \frac{1}{2}, 2+2s \Big)}_{=1} = 0.$$ Since the discriminant group of $2x^2 - 2y^2 + z^2$ is $16$ times the size of that of $x^2$, we conclude that the coefficient of $q^{(4n+7)/4}$ in the series $E_{3/2}(\tau;\tilde Q_2)_{(0,0,1/2)}$ is $$\frac{1}{4} \cdot \frac{2}{3/2} \cdot \begin{cases} -12H(4n+7): & n \equiv 0 \, (2); \\ 0: & n \equiv 1\, (2); \end{cases}$$ which implies the claim.
\end{proof}

Finally, the coefficients $a(n-r^2 / 4, (0,r/2))$ of the shadow $\vartheta$ are $$a(n-r^2 / 4, (0,r/2)) = \begin{cases} -8: & n-r^2 / 4 = 0; \\ -16: & r^2 - 4n \; \text{is a square}; \\ 0: & \text{otherwise}. \end{cases}$$ Comparing coefficients in $Q_{2,1,0} = E_2(\tau;Q)_0$ gives the formula:

\begin{lem} If $n \equiv 0 \, \bmod \, 4$, then $$8 \sigma_1(n/2) + 16 \sigma_1(n/4) = \sum_{r \in \mathbb{Z}} |\overline{\alpha}_2(n-r^2)| + 4 \sum_{r \, \mathrm{odd}} H(4n-r^2) + 2\sum_{r^2 - 4n = \square } \Big( |r| - \sqrt{r^2 - 4n} \Big) - \begin{cases} 4 (\sqrt{n} + 1): & n = \square; \\ 0: & \mathrm{otherwise}. \end{cases}$$  If $n \equiv 2\, \bmod \, 4$, then $$24 \sigma_1(n/2) = \sum_{r \in \mathbb{Z}} |\overline{\alpha}_2(n-r^2)| + 4 \sum_{r \, \mathrm{odd}} H(4n-r^2) + 2 \sum_{r^2 - 4n = \square} \Big( |r| - \sqrt{r^2 - 4n} \Big).$$ If $n \equiv 1\,  \bmod \, 2$, then $$4 \sigma_1(n) = \sum_{r \in \mathbb{Z}} |\overline{\alpha}_2(n - r^2)| + 2 \sum_{r \in \mathbb{Z}} \Big( |r| - \sqrt{r^2 - 4n} \Big) - \begin{cases} 4 (\sqrt{n} + 1) : & n = \square; \\ 0: & \mathrm{otherwise}; \end{cases}$$
\end{lem}

Since $$\sum_{r^2 - 4n = \square} \left( |r| - \sqrt{r^2 - 4n} \right) = 2 \sum_{d | n} \min(d, n/d) + \begin{cases} 4 \sqrt{n}: & n = \square; \\ 0: & \mathrm{otherwise}; \end{cases}$$ this implies proposition 2. \\

Corollary $3$ follows by algebraic manipulation from this and the identity (Corollary 1.1 of \cite{BL}) \begin{equation} |\overline{\alpha}_2(n)| = \begin{cases} r_3(n)/3 : & n \equiv 1,2 \, (4); \\ 8H(n) - r_3(n)/3: & n \equiv 0,3 \, (4); \end{cases} \end{equation} where $r_3(n)$ is the representation count of $n$ as a sum of three squares, using the identities $$\sum_r r_3(n-r^2) = 8 \sigma_1(n) - 32 \sigma_1(n/4)$$ (Jacobi's formula), and $$\sum_r (-1)^r r_3(n-r^2) = (-1)^{(n-1)/2} \cdot  4\sigma_1(n), \; \; n \; \mathrm{odd},$$ (which follows from Jacobi's formula) and the class number relations $$\sum_{r \in \mathbb{Z}} H(4n - r^2) = 2 \sigma_1(n) - 2 \lambda_1(n)$$ and $$\sum_{r \in \mathbb{Z}} H(n - r^2) = \frac{1}{3}\sigma_1(n) - \lambda_1(n), \; \; n \; \mathrm{odd},$$ the latter of which is due to Eichler. (See also Equation 1.3 of \cite{M}.)

\section{Proof of proposition 4}

This identity is derived analogously to Proposition 2, but it arises from the equality $Q_{2,2,0} = E_2$ for the dual Weil representation attached to the quadratic form $Q(x,y) = 2x^2 - y^2$. The discriminant group of this quadratic form has nonsquare order $8$, so $E_2(\tau;Q)$ is a true modular form; its $\mathfrak{e}_0$-component is $$E_2(\tau;Q)_0 = 1 - 18q - 34q^2 - 28q^3 - 66q^4 - 56q^5 - 60q^6 -  ... \in M_2(\Gamma_1(8))$$ in which the coefficient of $q^n$ is $$-8 \sigma_1(n,\chi) \cdot \Big( 2 + \frac{\chi(m)}{2^{\nu + 2}} \Big), \; \; n = 2^{\nu} m, \; m \; \mathrm{odd},$$ where $\chi$ is the Dirichlet character modulo $8$ given by $\chi(1) = \chi(7) = 1$, $\chi(3) = \chi(5) = -1$ and where $\sigma_1(n,\chi) = \sum_{d | n} \chi(n/d) d$. (This can be calculated using the coefficient formula of \cite{BK}, for example.) \\

The relevant components of the weight $3/2$ Eisenstein series for $Q_2(x,y,z) = 2x^2 - y^2 + 2z^2$ have Fourier series beginning \begin{align*} E_{3/2}(\tau;Q_2)_{(0,0,0)} &= 1 - 2q - 4q^2 - 4q^4 - 8q^5 - ... \\ E_{3/2}(\tau;Q_2)_{(0,0,1/2)} &= -2q^{1/2} - 4q^{3/2} - 4q^{5/2} - 8q^{7/2} - 6q^{9/2} - ... \\ E_{3/2}(\tau)_{(0,0,1/4)} = E_{3/2}(\tau)_{(0,0,3/4)} &= -4q^{7/8} - 4q^{15/8} - 4q^{23/8} - 8q^{31/8} - 4q^{39/8} - ... \end{align*} The general term is as follows:

\begin{lem} (i) The coefficient of $q^{n/2}$ in $E_{3/2}(\tau;Q_2)_{(0,0,1/2)}$ is $(-1/2)$ times the number of representations of $2n$ by the quadratic form $4a^2 + b^2 + c^2$. \\ (ii) The coefficient of $q^{n/8}$ in $E_{3/2}(\tau;Q_2)_{(0,0,1/4)}$ is $(-1/2)$ times the number of representations of $n$ by the quadratic form $4a^2 + 2b^2 + c^2$.
\end{lem}
\begin{proof} The components $E_{3/2}(\tau;Q_2)_{(0,0,1/2)}$ and $E_{3/2}(\tau;Q_2)_{(0,0,1/4)}$ are modular forms because the components $\mathfrak{e}_{(0,0,1/2)}$ and $\mathfrak{e}_{(0,0,1/4)}$ do not appear in the shadow $\vartheta(\tau).$ Once an equality between two modular forms has been conjectured (here, the components of $E_{3/2}(\tau;Q_2)$ and two theta series), it can always be proved by comparing a finite number of coefficients. In principle this could also be proven directly via the same argument as Lemma 7.
\end{proof}

The coefficient formula for the $\mathfrak{e}_0$-component of the index-$2$ series $$Q_{2,2,0}(\tau)_0 = \sum_{n=0}^{\infty} b(n) q^n$$ is now $$b(n) = \sum_{r \in \mathbb{Z}} c(n- r^2 / 8, (0,0,r/4)) + \frac{1}{8 \sqrt{2}} \sum_{r \in \mathbb{Z}} a(n-r^2 / 8, (0,0,r/4)) \Big( |r| - \sqrt{r^2 - 8n} \Big),$$ where $c(n,\gamma)$ is the coefficient of $q^n \mathfrak{e}_{\gamma}$ in the mock Eisenstein series above and $a(n,\gamma)$ is the coefficient of $q^{-n} \mathfrak{e}_{\gamma}$ in its shadow. Therefore, $$\sum_{r \in \mathbb{Z}} c(n-r^2 / 8, (0,0,r/4)) = -\sum_{r \in \mathbb{Z}} |\overline{\alpha}_2(n - 2r^2)| - \frac{1}{2}\sum_{r\, \mathrm{odd}} \Big( r_A(4n - 2r^2) + r_B(8n - r^2) \Big) + \begin{cases} 4: & 2n = \square; \\ 0: & \mathrm{otherwise}. \end{cases}$$ Here, $r_A(n)$ is the representation count of $n$ by $4a^2 + b^2 + c^2$ and $r_B(n)$ is the representation count of $n$ by $4a^2 + 2b^2 + c^2$. \\

\begin{rem} The generating function of the coefficients $\sum_{r \, \mathrm{odd}} r_A(n - 2r^2)$ is the difference of theta functions for the quadratic forms $4a^2 + b^2 + c^2 + 2d^2$ and $4a^2 + b^2 + c^2 + 4d^2$. In particular, $\sum_{n=0}^{\infty} \sum_{r \, \mathrm{odd}} r_A(4n - 2r^2) q^n$ is a modular form of weight $2$. By computing enough terms we can identify it as the eta product $$\sum_{n=0}^{\infty} \sum_{r \, \mathrm{odd}} r_A(4n-2r^2) q^n = 8q + 16q^2 + 16q^3 + 32q^4 + ... = \frac{8 \eta(2\tau)^3 \eta(4\tau) \eta(8\tau)^2}{\eta(\tau)^2}.$$ Similarly, $$\sum_{n=0}^{\infty} \sum_{r \, \mathrm{odd}} r_B(8n - r^2) q^n = 16q + 32 q^2 + 32q^3 + 64q^4 + ... = \frac{16 \eta(2\tau)^3 \eta(4\tau) \eta(8\tau)^2}{\eta(\tau)^2}.$$ The eta product $\frac{\eta(2\tau)^3 \eta(4 \tau) \eta(8\tau)^2}{\eta(\tau)^2}$ is one of the few such products with multiplicative coefficients, as classified by Martin \cite{Ma}, and its coefficient of $q^n$ is the twisted divisor sum $\sigma_1(n,\chi) = \sum_{d | n} \chi(n/d) d$ for the character $\chi(1) = \chi(7) = 1$, $\chi(3) = \chi(5) = -1$ mod $8$ that we consider throughout. Therefore, we can simplify the above sum to $$\sum_{r \in \mathbb{Z}} c(n-r^2 / 8, (0,0,r/4)) = -\sum_{r \in \mathbb{Z}} |\overline{\alpha}_2(n - 2r^2)| - 12 \sigma_1(n,\chi) + \begin{cases} 4: & 2n = \square; \\ 0: & \mathrm{otherwise}. \end{cases}$$
\end{rem}

The correction term $$\frac{1}{8 \sqrt{2}} \sum_{r \in \mathbb{Z}} a(n-r^2 / 8, (0,0,r/4)) \Big(|r| - \sqrt{r^2 - 8n}\Big) = -\sqrt{2} \sum_{\substack{r \in \mathbb{Z} \\ 2(r^2 - 8n) = \square}} \Big(|r| - \sqrt{r^2 - 8n} \Big) \times \begin{cases} 1: & r^2 \ne 8n; \\ 1/2: & r^2 = 8n; \end{cases}$$  is less straightforward to calculate than the corresponding term in the proof of Proposition 2 because the discriminant $8$ is not square. Following section $7$ of \cite{W2}, this term can be calculated by finding minimal solutions to the Pell-type equation $a^2 - 8b^2 = -64n$.\\

 The true Pell equation $a^2 - 8b^2 = 1$ has fundamental solution $a = 3, b = 1$. We let $\mu_i = a_i + b_i \sqrt{8}$, $i \in \{1,...,N\}$ denote the representatives of orbits of elements in $\mathbb{Z}[\sqrt{2}]$ up to conjugation having norm $2n$ and minimal positive trace; then $$\frac{1}{\sqrt{2}} \sum_{2(r^2 - 8n) = \square} \Big[ \Big( |r| - \sqrt{r^2 - 8n} \Big) \times \begin{cases} 1: & r^2 - 8n \ne 0; \\ 1/2: & r^2 - 8n = 0; \end{cases} \Big] = -4 \sum_{i=1}^N \Big( |b_i| - \frac{a_i}{2} \Big) \times \begin{cases} 2: & \overline{\mu_i} / \mu_i \notin \mathbb{Z}[\sqrt{2}]; \\ 1: & \overline{\mu_i} / \mu_i \in \mathbb{Z}[\sqrt{2}]. \end{cases}$$ 

Since $\mathbb{Q}(\sqrt{2})$ has class number one, these orbits correspond to the ideals of $\mathbb{Z}[\sqrt{2}]$ with ideal norm $2n$ and we can write $$- \sqrt{2} \sum_{2 (r^2 - 8n) = \square} \Big[ \Big(|r| - \sqrt{r^2 - 8n} \Big)  \times \begin{cases} 1 : & r^2 - 8n \ne 0; \\ 1/2 : & r^2 - 8n = 0; \end{cases} \Big] = 4 \sum_{N(\mathfrak{a}) = 2n} \Big( |b| - a \Big),$$ where $\mathfrak{a}$ runs through the ideals of $\mathbb{Z}[\sqrt{2}]$ of norm $2n$ and $a + b \sqrt{2} \in \mathfrak{a}$ is a generator with minimal $a > 0$. \\

Comparing coefficients between $Q_{2,2,0}(\tau)_0$ and $E_2(\tau;Q)_0$ results in the identity \begin{align*} \sum_{r \in \mathbb{Z}} |\overline{\alpha}_2(n - 2r^2)| &= 2 \sigma_1(n,\chi) \cdot \Big( 2 + \frac{\chi(m)}{2^{\nu}} \Big) + 4 \sum_{N(\mathfrak{a}) = 2n} \Big( |b| - a \Big) + \begin{cases} 4: & 2n = \square; \\ 0: & \mathrm{otherwise}; \end{cases} \end{align*} as claimed, where $a + b \sqrt{2} \in \mathfrak{a}$ is a generator with minimal $a > 0$.

\begin{ex} Let $n = 7$. The ideals of $\mathbb{Z}[\sqrt{2}]$ of norm $14$ are $(4 \pm \sqrt{2})$ and the trace $8$ is minimal within both ideals. The left side of lemma 11 is $$\sum_{r \in \mathbb{Z}} |\overline{\alpha}_2(7 - 2r^2)| = |\overline{\alpha}_2(7)| + 2 \cdot |\overline{\alpha}_2(5)| = 24,$$ while the right side is $$2 \sigma_1(7,\chi) (2 + \chi(7)) + 4 (1 - 4) + 4 (1-4) = 48 - 12 - 12 = 24.$$
\end{ex}

\begin{rem}
When $n = p$ is a prime that remains inert in $\mathbb{Z}[\sqrt{2}]$ (i.e. $\chi(p) = -1$) this identity simplifies to $$\sum_{r \in \mathbb{Z}} |\overline{\alpha}_2(p - 2r^2)| = 2(p-1).$$
\end{rem}

To prove Corollary 5, we again use equation (4). In the first case, we start from $$\sum_r |\overline{\alpha}_2(4n - 2r^2)| = 8\sum_{r \in \mathbb{Z}} H(4n - 2r^2) - \frac{1}{3} \sum_{r \in \mathbb{Z}} (-1)^r r_3(4n - 2r^2).$$ The generating function $$\sum_{n=0}^{\infty} \sum_{r \in \mathbb{Z}} (-1)^r r_3(4n - 2r^2) q^n = 1 - 18q - 34q^2 - 28q^3 - 66q^4 - 56q^5 - 60q^6 - ...$$ is a difference of theta functions and therefore a modular form of level $8$; and we identify it as $E_2(\tau)_0$, giving the identity $$\sum_{r \in \mathbb{Z}} (-1)^r r_3(4n - 2r^2) = -8 \sigma_1(n,\chi) \cdot \Big( 2 + \frac{\chi(m)}{2^{\nu + 2}} \Big), \; \; n = 2^{\nu} m, \; m \; \mathrm{odd}.$$ Therefore, \begin{align*} \sum_{r \in \mathbb{Z}} H(4n - 2r^2) &= -\frac{1}{3} \sigma_1(n,\chi) \cdot \Big( 2 + \frac{\chi(m)}{2^{\nu+2}} \Big) + \frac{1}{4} \sigma_1(4n,\chi) \cdot \Big( 2 + \frac{\chi(m)}{2^{\nu + 2}} \Big) + \frac{1}{2} \sum_{N(\mathfrak{a}) = 8n} \Big( |b| - a \Big) \\ &= \frac{2}{3} \sigma_1(n,\chi) \cdot \Big( 2 + \frac{\chi(m)}{2^{\nu + 2}} \Big) + \frac{1}{2} \sum_{N(\mathfrak{a}) = 8n} \Big(|b| - a \Big). \end{align*}

In the second case, for odd $n$, $$\sum_r |\overline{\alpha}_2(2n - 2r^2)| = 8\sum_{r \in \mathbb{Z}} H(2n - 2r^2) + \frac{1}{3} \sum_{r \in \mathbb{Z}} (-1)^r r_3(2n - 2r^2).$$ Here, $$\sum_{r \in \mathbb{Z}} (-1)^r r_3(2n - 2r^2) = (8 + 2 \chi(n)) \cdot \sigma_1(n,\chi), \; \; n \; \mathrm{odd},$$ so \begin{align*} \sum_{r \in \mathbb{Z}} H(2n - 2r^2) &= -\frac{8 + 2 \chi(n)}{24} \sigma_1(n,\chi) + \frac{1}{4} \sigma_1(2n,\chi) \cdot \Big( 2 + \frac{\chi(n)}{2} \Big) + \frac{1}{2} \sum_{N(\mathfrak{a}) = 4n} \Big( |b| - a \Big) \\ &= \frac{4 + \chi(n)}{6} \sigma_1(n,\chi) + \frac{1}{2} \sum_{N(\mathfrak{a}) = 4n} \Big( |b| - a\Big). \end{align*}

In the third case, for odd $n$, $$\sum_r |\overline{\alpha}_2(n-2r^2)| = 8 \sum_{r \in \mathbb{Z}} H(n - 2r^2) - \frac{\chi(n)}{3} \sum_{r \in \mathbb{Z}} (-1)^r r_3(n - 2r^2),$$ where $$\sum_{r \in \mathbb{Z}} (-1)^r r_3(n - 2r^2) = \begin{cases} 6 \sigma_1(n,\chi): & n \equiv 1 \, (8); \\ -2 \sigma_1(n,\chi): & n \equiv 3 \, (8); \\ 2 \sigma_1(n,\chi): & n \equiv 5 \, (8); \\ -6 \sigma_1(n,\chi): & n \equiv 7 \, (8); \end{cases}$$ and therefore $$\sum_{r \in \mathbb{Z}} H(n - 2r^2) = \frac{2 + \chi(n)}{6} \sigma_1(n,\chi)  + \frac{1}{2} \sum_{N(\mathfrak{a}) = 2n} \Big( |b| - a \Big).$$

\textbf{Acknowledgments:} I thank the reviewer for many suggestions regarding the exposition of this note.

\bibliographystyle{plainnat}
\bibliofont
\bibliography{\jobname}

\begin{thebibliography}{18}
\providecommand{\natexlab}[1]{#1}
\providecommand{\url}[1]{\texttt{#1}}
\expandafter\ifx\csname urlstyle\endcsname\relax
  \providecommand{\doi}[1]{doi: #1}\else
  \providecommand{\doi}{doi: \begingroup \urlstyle{rm}\Url}\fi

\bibitem[Borcherds(1998)]{B}
Richard Borcherds.
\newblock Automorphic forms with singularities on {G}rassmannians.
\newblock \emph{Invent. Math.}, 132\penalty0 (3):\penalty0 491--562, 1998.
\newblock ISSN 0020-9910.
\newblock \doi{10.1007/s002220050232}.
\newblock URL \url{http://dx.doi.org/10.1007/s002220050232}.

\bibitem[Bringmann and Kane(2013)]{BK2}
Kathrin Bringmann and Ben Kane.
\newblock Sums of class numbers and mixed mock modular forms.
\newblock Preprint, 2013.
\newblock URL \url{arxiv:1305.0112}.

\bibitem[Bringmann and Lovejoy(2009)]{BL}
Kathrin Bringmann and Jeremy Lovejoy.
\newblock Overpartitions and class numbers of binary quadratic forms.
\newblock \emph{Proc. Natl. Acad. Sci. USA}, 106\penalty0 (14):\penalty0
  5513--5516, 2009.
\newblock ISSN 1091-6490.
\newblock \doi{10.1073/pnas.0900783106}.
\newblock URL \url{http://dx.doi.org/10.1073/pnas.0900783106}.

\bibitem[Bringmann et~al.(2017)Bringmann, Folsom, Ono, and Rolen]{O}
Kathrin Bringmann, Amanda Folsom, Ken Ono, and Larry Rolen.
\newblock \emph{Harmonic Maass forms and mock modular forms: theory and
  applications}, volume~64 of \emph{American Mathematical Society Colloquium
  Publications}.
\newblock American Mathematical Society, Providence, Rhode Island, 2017.

\bibitem[Brown et~al.(2008)Brown, Calkin, Flowers, James, Smith, and
  Stout]{six}
Brittany Brown, Neil Calkin, Timothy Flowers, Kevin James, Ethan Smith, and Amy
  Stout.
\newblock Elliptic curves, modular forms, and sums of {H}urwitz class numbers.
\newblock \emph{J. Number Theory}, 128\penalty0 (6):\penalty0 1847--1863, 2008.
\newblock ISSN 0022-314X.
\newblock \doi{10.1016/j.jnt.2007.10.008}.
\newblock URL \url{http://dx.doi.org/10.1016/j.jnt.2007.10.008}.

\bibitem[Bruinier and Kuss(2001)]{BK}
Jan Bruinier and Michael Kuss.
\newblock Eisenstein series attached to lattices and modular forms on
  orthogonal groups.
\newblock \emph{Manuscripta Math.}, 106\penalty0 (4):\penalty0 443--459, 2001.
\newblock ISSN 0025-2611.
\newblock \doi{10.1007/s229-001-8027-1}.
\newblock URL \url{http://dx.doi.org/10.1007/s229-001-8027-1}.

\bibitem[Cowan et~al.(2017)Cowan, Katz, and White]{CKW}
Raemeon Cowan, Daniel Katz, and Lauren White.
\newblock A new generating function for calculating the {I}gusa local zeta
  function.
\newblock \emph{Adv. Math.}, 304:\penalty0 355--420, 2017.
\newblock ISSN 0001-8708.
\newblock \doi{10.1016/j.aim.2016.09.003}.
\newblock URL \url{http://dx.doi.org/10.1016/j.aim.2016.09.003}.

\bibitem[Hirzebruch and Zagier(1976)]{HZ}
Friedrich Hirzebruch and Don Zagier.
\newblock Intersection numbers of curves on {H}ilbert modular surfaces and
  modular forms of {N}ebentypus.
\newblock \emph{Invent. Math.}, 36:\penalty0 57--113, 1976.
\newblock ISSN 0020-9910.
\newblock \doi{10.1007/BF01390005}.
\newblock URL \url{http://dx.doi.org/10.1007/BF01390005}.

\bibitem[Imamo\u{g}lu et~al.(2014)Imamo\u{g}lu, Raum, and Richter]{ORR}
\"Ozlem Imamo\u{g}lu, Martin Raum, and Olav Richter.
\newblock Holomorphic projections and {R}amanujan's mock theta functions.
\newblock \emph{Proc. Natl. Acad. Sci. USA}, 111\penalty0 (11):\penalty0
  3961--3967, 2014.
\newblock ISSN 1091-6490.
\newblock URL \url{https://doi.org/10.1073/pnas.1311621111}.

\bibitem[Kaneko and Zagier(1995)]{KZ}
Masanobu Kaneko and Don Zagier.
\newblock A generalized {J}acobi theta function and quasimodular forms.
\newblock In \emph{The moduli space of curves ({T}exel {I}sland, 1994)}, volume
  129 of \emph{Progr. Math.}, pages 165--172. Birkh\"auser Boston, Boston, MA,
  1995.
\newblock \doi{10.1007/978-1-4612-4264-2_6}.
\newblock URL \url{http://dx.doi.org/10.1007/978-1-4612-4264-2_6}.

\bibitem[Lovejoy(2008)]{L}
Jeremy Lovejoy.
\newblock Rank and conjugation for a second {F}robenius representation of an
  overpartition.
\newblock \emph{Ann. Comb.}, 12\penalty0 (1):\penalty0 101--113, 2008.
\newblock ISSN 0218-0006.
\newblock \doi{10.1007/s00026-008-0339-0}.
\newblock URL \url{http://dx.doi.org/10.1007/s00026-008-0339-0}.

\bibitem[Martin(1996)]{Ma}
Yves Martin.
\newblock Multiplicative {$\eta$}-quotients.
\newblock \emph{Trans. Amer. Math. Soc.}, 348\penalty0 (12):\penalty0
  4825--4856, 1996.
\newblock ISSN 0002-9947.
\newblock \doi{10.1090/S0002-9947-96-01743-6}.
\newblock URL \url{http://dx.doi.org/10.1090/S0002-9947-96-01743-6}.

\bibitem[Mertens(2014)]{M}
Michael Mertens.
\newblock Mock modular forms and class number relations.
\newblock \emph{Res. Math. Sci.}, 1:\penalty0 Art. 6, 16, 2014.
\newblock ISSN 2197-9847.
\newblock \doi{10.1186/2197-9847-1-6}.
\newblock URL \url{http://dx.doi.org/10.1186/2197-9847-1-6}.

\bibitem[Skoruppa(1985)]{Sk}
Nils-Peter Skoruppa.
\newblock \emph{\"Uber den {Z}usammenhang zwischen {J}acobiformen und
  {M}odulformen halbganzen {G}ewichts}, volume 159 of \emph{Bonner
  Mathematische Schriften [Bonn Mathematical Publications]}.
\newblock Universit\"at Bonn, Mathematisches Institut, Bonn, 1985.
\newblock Dissertation, Rheinische Friedrich-Wilhelms-Universit\"at, Bonn,
  1984.

\bibitem[Sturm(1987)]{St}
Jacob Sturm.
\newblock On the congruence of modular forms.
\newblock In \emph{Number theory ({N}ew {Y}ork, 1984--1985)}, volume 1240 of
  \emph{Lecture Notes in Math.}, pages 275--280. Springer, Berlin, 1987.
\newblock \doi{10.1007/BFb0072985}.
\newblock URL \url{https://doi.org/10.1007/BFb0072985}.

\bibitem[Williams(2017{\natexlab{a}})]{W1}
Brandon Williams.
\newblock Vector-valued {E}isenstein series of small weight.
\newblock To appear in \emph{Int. J. Number Theory}, 2017{\natexlab{a}}.
\newblock URL \url{arxiv:1706.03738}.

\bibitem[Williams(2017{\natexlab{b}})]{W2}
Brandon Williams.
\newblock Poincar\'e square series of small weight.
\newblock To appear in \emph{Ramanujan J.}, 2017{\natexlab{b}}.
\newblock URL \url{arxiv:1707.06582}.

\bibitem[Williams(2018)]{W3}
Brandon Williams.
\newblock Vector-valued {H}irzebruch-{Z}agier series and class number sums.
\newblock \emph{Res. Math. Sci.}, 5\penalty0 (2):\penalty0 Paper No. 25, 13,
  2018.
\newblock ISSN 2522-0144.
\newblock \doi{10.1007/s40687-018-0142-4}.
\newblock URL \url{https://doi.org/10.1007/s40687-018-0142-4}.

\end{thebibliography}

\end{document}